\documentclass{amsart}

\usepackage{amsmath,color}
\usepackage{amssymb}
\usepackage{amsfonts}
\usepackage{graphicx}
\usepackage{mathrsfs}
\usepackage{amscd}
\usepackage[all]{xy}
\usepackage{hyperref}
\usepackage{amsthm}
\usepackage{fancyhdr}
\usepackage{verbatim}
\usepackage{amscd}
\usepackage{mathtools}

\title[On the Growth of Deviations]{On the Growth of Deviations}

\subjclass[2010]{Primary: 13D02; Secondary: 16E45, 13D40, 16S37, 05C25, 05C38. }

\keywords{ Deviations, DG algebra resolution,  Poincar\'e series, Hilbert series, Betti numbers, Koszul algebra, Golod ring, initial ideal, lex-segment ideal, edge ideal.}

\author[Adam Boocher]{Adam Boocher}
\address{Adam Boocher \\ School of Mathematics \\ University of Edinburgh \\James Clerk Maxwell Building, Mayfield Road \\Edinburgh EH9 3JZ, Scotland}
\email{adam.boocher@ed.ac.uk}
 
 \author[Alessio D'Al\`{i}]{Alessio D'Al\`{i}}
\address{Alessio D'Al\`{i} \\ Dipartimento di Matematica \\ Universit\`{a} degli Studi di Genova \\Via Dodecaneso 35 \\16146 Genova, Italy}
\email{dali@dima.unige.it}
 
 \author[Elo\'isa Grifo]{Elo\'isa Grifo}
\address{Elo\'isa Grifo \\Department of Mathematics \\  University of Virginia \\141 Cabell Drive, Kerchof Hall \\Charlottesville, VA 22904, USA}
\email{er2eq@virginia.edu}

\author[J. Monta\~{n}o]{Jonathan Monta\~{n}o}
\address{Jonathan Monta\~{n}o \\ Department of Mathematics \\University of Kansas \\405 Snow Hall, 1460 Jayhawk Blvd \\Lawrence, KS 66045}
\email{jmontano@ku.edu}

\author[Alessio Sammartano]{Alessio Sammartano}
\address{Alessio Sammartano \\ Department of Mathematics \\ Purdue University \\150 North University Street \\West Lafayette, IN 47907, USA}
\email{asammart@purdue.edu}

\newtheorem{thm}{\bf Theorem}[section]

\newtheorem{prop}[thm]{\bf Proposition}

\theoremstyle{definition}

\theoremstyle{remark}
\newtheorem{remark}[thm]{\bf Remark}

\newtheorem{question}[thm]{\bf Question}
\newtheorem{example}[thm]{\bf Example}
\numberwithin{equation}{section}

\newcommand{\be}[2]{\beta_{#1,#2}}

\def \edim{{\operatorname{edim\, }}}

\def\Der{\operatorname{Der}}

\def\Tor{\operatorname{Tor}}
\def\Hom{\operatorname{Hom}}
\def\Ext{\operatorname{Ext}}

\def \initial{{\operatorname{in}}}
\def \gin{{\operatorname{gin}}}
\def \ch{{\operatorname{char }}}

\def \Card{{\operatorname{Card}}}

\DeclareMathOperator{\HS}{HS}

\DeclareMathOperator{\HF}{HF}
\DeclareMathOperator{\Gl}{GL}

\def \f1{\mathbf{1}}

\def\xi{x}

\def\zi{z}

\def\ls{\leqslant}
\def\gs{\geqslant}

\def\fm{\mathfrak{m}}
\def\ee{\varepsilon}
\def\fn{\mathfrak{n}}

\def \CC{\mathbb C}
\def \RR{\mathbb R}

\def \NN{\mathbb N}

\def \G{\mathcal G}

\begin{document}

\begin{abstract}
The deviations of a graded algebra are a sequence of integers that determine the Poincar\'e series of its residue field and arise as the number of generators of certain DG algebras. In a sense, deviations measure how far a ring is from being a complete intersection.  In this paper we study extremal deviations among those of algebras with a fixed Hilbert series. 
In this setting, we prove that, like the Betti numbers, deviations do not decrease when passing to an initial ideal and are maximized by the Lex-segment ideal. We also prove that deviations grow exponentially for Golod rings and for certain quadratic monomial algebras. 
\end{abstract}

\maketitle

\section{Introduction} 
\label{SectionPreliminaries}
\noindent Let $R = S/I$, where $S=k[T_1,\ldots,T_n]$ is a polynomial ring over a field $k$ and $I\subseteq (T_1,\ldots,T_n)^2$ is a homogeneous ideal.  With the pair $(S,R)$ we can associate two important sets of Betti numbers, namely 
$$\beta_i^S(R) = \dim_k \Tor_i^S(R,k) \qquad \textrm{ and } \qquad \beta_i^R(k) = \dim_k \Tor_i^R(k,k).$$
Equivalently, these numbers are given by the ranks of the free modules appearing in minimal free resolutions of $R$ over $S$ and $k$ over $R$ respectively.  To avoid confusion, and to emphasize that the latter sequence is typically infinite, we will often refer instead to the Poincar\'e series $P_k^R(\zi)= \sum_{i=1}^\infty\beta_{i}^R(k)\zi^i$.  
Much work has been devoted to studying the growth of Betti numbers for various classes of rings.  For instance, by the Auslander-Buchsbaum-Serre Theorem, $P_k^R(\zi)$ is a polynomial if and only if $R$ is regular, i.e., a polynomial ring. 
 In the next simplest case,  the sequence $\{\beta_i^R(k)\}$ has polynomial growth if  $R$ is a complete intersection
 (cf. \cite{Tate}).
Finally, if $R$ is not a complete intersection, then the $\beta_i^R(k)$ have  exponential growth (cf. \cite{AvramovBetti}).  

In this paper we study Betti numbers from the point of view of a related invariant: the set of deviations of $R$.  Since $P_k^R(\zi)$ has integer coefficients and  constant term equal to 1, there exist uniquely determined integers $\ee_{i}=\ee_{i}(R)$,
called the {deviations} of $R$,
such that the following infinite product expansion holds (cf. \cite[7.1.1]{Avramov6Lectures}):

\begin{equation}\label{PoincareDeviations}
P_k^R(\zi)=\prod_{i= 1}^\infty \frac{(1+\zi^{2i-1})^{\ee_{2i-1}}}{(1-\zi^{2i})^{\ee_{2i}}}.
\end{equation}

In the sense explained in Remark \ref{Remark about complexity}, the deviations measure the complexity of $R$.
They play a crucial role in Avramov's proof that the complete intersection property localizes (cf. \cite{AvramovCI}).
In addition, the $\ee_i(R)$ appear naturally as the number of generators of degree $i$ in an acyclic closure of $k$ over $R$,
  as the number of generators of degree $i-1$ in a minimal model of $R$ over $S$,
  and as the ranks of the components of the homotopy Lie algebra of $R$, cf. Section \ref{Preliminaries}.

Our first results concern the deviations of algebras with fixed Hilbert function. 
In Theorem \ref{degenerations have bigger deviations} we prove that, for any term order $<$, the algebra presented by $\initial_< I$ has larger deviations than $R$. 
Moreover, we show in Theorem \ref{Main Theorem on Lex} that the algebra presented by the {lex-segment} ideal 
has the largest deviations. 
This is a generalization of a result of Peeva which states that the lex-segment ideal attains the largest values of $P_k^{R}(\zi)$ among all $I$ with the same Hilbert function (cf. \cite{Peeva0Borel}). 
Peeva's theorem in turn relies on its analogue for $\beta_i^S(R)$, which is due to Bigatti, Hulett, and Pardue (cf. \cite{Bigatti, Hulett, Pardue}). 
Our methods use their techniques to provide a proof in all characteristics. In characteristic zero, we also present a simpler proof using Golod rings. 
Golod rings are those rings $R$ whose Poincar\'e series is maximal among all ideals with a fixed embedding dimension and set of Betti numbers $\beta_i^S(R)$;
see \cite[6.16]{McculloughPeeva} for a  list of Golod rings.
In particular, if $R$ is a Golod ring, then $P_k^R(\zi)$ is a rational function of $\zi$ determined by the  $\beta_i^S(R)$ and $\edim R$.

In the second half of the paper we turn to an analysis of the asymptotic behavior of deviations.  
A result of Babenko shows that the radius of convergence of the generating function of the $\ee_i(R)$ coincides with that of $P_k^R(\zi)$,
in the non-trivial case when $\ee_i(R) \ne 0$ infinitely often (cf. \cite[Theorem 1]{Babenko}).
Furthermore, there exist subsequences of $\{\ee_i(R)\}$
 which are not too sparse and have an exponential lower bound,
provided that $R$ is not a complete intersection (cf. \cite{AvramovBetti,FelixThomas}).
In Theorem \ref{goloddev} and Theorem \ref{clawfreedev} we prove that, for Golod rings and for certain Koszul algebras, the sequence of  deviations 
is asymptotically equal to $\{\frac{b\rho^i}{i}\}_{i\gs 1}$ for some $\rho>1, b\in \mathbb{N}$.

\section{Preliminaries}\label{Preliminaries}
We now introduce the notions of acyclic closures and minimal models, and how they encode the deviations.  We set up some notation and generally follow \cite{Avramov6Lectures} throughout the paper. Although the results therein are proved for  local rings, they can be extended to the graded case with minor changes in the proofs (cf. \cite[8.3.7]{Avramov6Lectures}, see also \cite{AvramovConcaIyengar}).

Let $A= \oplus_{i\gs 0}A_i$ be a DG algebra
such that $A_0$ is a Noetherian standard graded algebra over a field.
Assume that $A$ is bi-graded with internal and homological degree;
we will denote homological degree by $|\cdot|$ and internal degree by $\deg(\cdot)$.
For the differential we have $|\partial|=-1$ and $\deg(\partial)=0$.
It follows from the Leibniz rule for $\partial$ that
the cycles form a bi-graded subalgebra $Z(A)$ and 
the boundaries form a bi-graded two-sided ideal $B(A)\subseteq Z(A)$, 
and therefore the homology $H(A)=Z(A)/B(A)$ also has a bi-graded algebra structure.
If $\zeta \in Z_i(A)_j$ is a bi-homogeneous cycle,
let $y$ be a new variable of bi-degree $(i+1,j)$ and
denote by $A\langle y\rangle$ the unique, up to isomorphism, DG algebra extension of $A$ such that the differential satisfies $\partial(y)=\zeta$, 
where $y$ is an exterior variable if $i$ is even and a  divided power variable if $i$ is odd (cf. \cite[6.1.1]{Avramov6Lectures}).
For any proper bi-homogeneous ideal $J\subseteq A_0$ we can construct a semi-free DG algebra resolution $A\langle Y \rangle $ of the factor ring $A_0/J$,
where $Y = \bigcup_{i\gs 1 } Y_i$ and  $Y_i$ is a finite set of bi-homogeneous variables of homological degree $i$ such that $\partial(Y_1)$ minimally generates $J$ modulo $\partial(A_1)$ and the homology classes 
of each $\partial(Y_{n+1})$ minimally generate the $A_0$-module $H_n(A\langle Y_{\ls n} \rangle)$.
Here $Y_{\ls n}$ is short for $\bigcup_{i\ls n }Y_i$.
A DG algebra $A\langle Y \rangle$ obtained in this way is called an {\bf acyclic closure} of $A_0/J$ over the DG algebra $A$ and the variables $Y$ are called {\bf $\boldsymbol{\Gamma}$-variables}.
Notice that  the $i$-th homology of $A\langle Y_{\ls n} \rangle$
vanishes for $0<i<n$.

A theorem proved independently by Gulliksen and Schoeller (cf. \cite{Gulliksen}, \cite{Schoeller}) states that an acyclic closure $R\langle Y \rangle$ of $k$ over the DG algebra $A=A_0=R$ is in fact a minimal resolution of $k$ over $R$,
from which one can easily deduce that the $i$-th deviation is equal to the number of $\Gamma$-variables in homological degree $i$:
$$
\varepsilon_{i}(R)= \Card(Y_{i}).
$$
It follows  in particular that $\ee_i(R)\gs 0$. 
Note that $R\langle Y_1\rangle$ is the Koszul complex $K^R$ of $R$ with respect to $\fm $.
Basic homological properties of $R$ can be characterized in terms of the vanishing of deviations, 
as illustrated in the following discussion,
which explains
the use of the word ``deviation'' (cf. \cite[7.3]{Avramov6Lectures}).

\begin{remark}\label{Remark about complexity}
The first deviation $\ee_{1}(R)$ is equal to the embedding dimension of $R$, 
thus 
the following conditions are equivalent:
(i) $R$ is a field,
(ii) $\ee_{1}(R)=0$,
(iii) $\ee_{i}(R)=0$ for every $i \gs 1$.
The exactness of the Koszul complex $K^R$ gives the following equivalent conditions:
(i) $R$ is regular,
(ii) $\ee_{2}(R)=0$,
(iii) $\ee_{i}(R)=0$ for every $i \gs 2$.
Finally, theorems of Assmus \cite{Assmus}, Halperin \cite{Halperin} and Tate \cite{Tate} give the following equivalent conditions:
(i) $R$ is a complete intersection,
(ii) $\ee_{3}(R)=0$,
(iii) $\ee_{i}(R)=0$ for every $i \gs 3$,
(iv) $\ee_{i}(R)=0$ for some $i\gs 3$.
\end{remark}

Following the notation in the construction of acyclic closures, denote by $\mathfrak{p}$ the homogeneous maximal ideal of $A_0$. 
Given $J\subseteq \mathfrak{p}^2$ a homogeneous ideal and starting from $A_0$, 
we can build another bi-graded DG algebra free resolution $A_0[X]$ of $A_0/J$ over $A_0$ by following the same steps above, but adjoining polynomial variables (instead of divided powers) to kill cycles in odd degrees.  
A DG algebra resolution obtained in this way is called a {\bf minimal model} of $A_0/J$ over $A_0$. 
If we do not require  $\partial(X_{n+1})$ to  generate $H_n(A_0[ X_{\ls n} ])$ minimally for each $n \gs 1$,
then the resolution obtained is simply called a {\bf model} of $A_0/J$ over $A_0$. 
Analogously to minimal free resolutions,
an equivalent condition for a model to be minimal is that 
$\partial(X_1)\subseteq \mathfrak{p}^2$ and
$\partial(X_{n+1})\subseteq \mathfrak{p}X_n+ \sum_{i=1}^{n-1} X_i (A_0[X])_{n-i}$ for $n\gs 1$ (cf. \cite[7.2.2]{Avramov6Lectures}). 
Minimal models always exist and are unique up to isomorphism (cf. \cite[7.2.4]{Avramov6Lectures}). 
If $J$ is a complete intersection  or if $A_0$ is of equal characteristic zero, 
then $A_0\langle Y\rangle$ and $A_0[X]$ are isomorphic as DG algebras. However they differ in general  (cf. \cite[6.1.10]{Avramov6Lectures}). 

The following theorem due to Avramov shows that minimal models also carry information about deviations \cite[7.2.5, 7.2.6]{Avramov6Lectures}. 
We include here the statement in the standard graded case:

\begin{thm}[Avramov]\label{Avramov}
Let $S=k[T_1,\ldots,T_n]$ and $\fn=(T_1,\ldots,T_n)$. 
Let $I$ be a homogeneous ideal such that $I\subseteq \fn^2$ and $R=S/I$. 
Let $S[X]$ be a model of $R$ over $S$. Then $\Card(X_{i})\gs \varepsilon_{i+1}(R)$ for every $i\gs 1$.
Furthermore, equality occurs for every $i$ if and only if $S[X]$ is a minimal model of $R$ over $S$.
\end{thm}
Hence, it is possible to compute $\ee_i(R)$ by computing a minimal model of $R$ over $S$. Some applications of this fact and a finer analysis of the sequence of deviations for edge ideals of special graphs are given in \cite{BDGMS}.

Finally, a third context where deviations arise naturally is that of homotopy Lie algebras.
The {\bf homotopy Lie algebra} of $R$ is the graded Lie algebra over $k$ 
$$
\pi(R) = H (\Der^\gamma_R(R\langle Y \rangle, R\langle Y\rangle))
$$
where $R\langle Y\rangle$ is an acyclic closure of $k$ over $R$ and $\Der^\gamma_R(R\langle Y \rangle, R\langle Y\rangle)$ denotes the DG module of $R$-linear $\Gamma$-derivations (cf.\cite[6.2.2]{Avramov6Lectures}).
Its universal enveloping algebra is $\Ext_R(k,k)$ and the dimension of the graded component
$\pi^i(R)$ is $\ee_i(R)$.
A result  due to Avramov and L\"ofwall  states that 
$R$ is Golod if and only if $\pi^{\gs 2}(R)$ is
the free graded Lie algebra generated  by the  vector space $\Hom_k(\Sigma H^R_{\gs 1},k)$,
where $H^R$ denotes the homology of the Koszul complex of $R$.
We refer to \cite[Chapter 10]{Avramov6Lectures} for more details on the subject.

\section{Extremal deviations}\label{Extremal}

The goal of this section is to study extremal deviations among ideals with a given Hilbert series.  
The first main result of this section is Theorem \ref{degenerations have bigger deviations}, where we show that the deviations of $S/I$ are at most equal to those of $S/\initial_{<}(I)$ for any term order $<$. 
The second main result is Theorem \ref{Main Theorem on Lex}, where we show that among all rings $S/I$ with the same Hilbert series, the ring $S/L$ has the largest deviations, where $L$ denotes the lex-segment ideal of $S/I$.  
We modify a standard deformation argument and apply it to minimal models.

A non-negative integer function $\omega$ on the set $\{T_1,\ldots,T_n\}$ is called a {\bf weight} for $S$;
we can extend it to arbitrary monomials in $S$ by
$\omega(T_1^{v_1}\cdots T_n^{v_n})=v_1\omega(T_1)+\cdots + v_n\omega(T_n) $. 
Given a polynomial $f\in S$, we denote the highest weight of a monomial in the support of $f$ by $\omega(f)$ and the sum of the terms of $f$ with weight equal to $\omega(f)$ by $\initial_{\omega} (f)$.
 Let $t$ be a new variable and
 define the ideal

$$
I_\omega=\Big(\, t^{\omega(f)}f\big(t^{-w(T_1)}T_1,\,\ldots,\, t^{-w(T_n)}T_n\big)\,|\, f\in I \Big)\subseteq S[t].
$$

Notice that upon setting $t=0$ in $I_\omega$ we obtain $\initial_\omega I:=(\initial_\omega(f)\,|\, f\in I)$ whereas upon setting $t=1$ we obtain $I$. 
If $<$ is a  term order, there is always a weight $\omega$ so that $\initial_\omega I= \initial_< I$ (cf. \cite[3.1.2]{HerzogHibi}). 

\begin{thm}\label{degenerations have bigger deviations}
Let $S=k[T_1,\ldots,T_n]$, $\fn=(T_1,\ldots,T_n)$, and $\omega$ be a weight for $S$.
  If $I\subseteq \fn^2$ is a homogeneous ideal, 
  then  we have
$\varepsilon_i (S/I) \ls \varepsilon_i (S/\initial_\omega I)$ for every $i\gs 1$.
In particular, if $<$ is a term order on $S$ then 
$\varepsilon_i (S/I) \ls \varepsilon_i (S/\initial_< I)$ for every $i\gs 1$.
\end{thm}

\begin{proof}
The result holds for $i =1$ as $\varepsilon_1(S/I)=n=\varepsilon_1 (S/\initial_\omega I)$.  
We  extend  $\omega$ to a weight $\omega'$ for $S[t]$ by setting $\omega'(T_i)=\omega(T_i)$ and $\omega'(t)=1$;
then  $I_\omega$ is a  homogeneous ideal of $S[t]$ under the grading induced by $\omega'$. 
 Let $S[t][X]$ be a minimal model of $S[t]/I_\omega$ over $S[t]$. 
Since  $t$ and $t-1$ are regular on $S[t]/I_\omega$ by \cite[3.2.5]{HerzogHibi}, 
it follows from \cite[3.2.6]{HerzogHibi} that $S[t][X]\otimes_{S[t]} S[t]/(t)$ and $S[t][X]\otimes_{S[t]} S[t]/(t-1)$ are free $S$-resolutions of $S/\initial_\omega I$ and $S/I$ respectively, 
and hence models over $S$, as they inherit the DG algebra structure. 

Since $S$ is positively graded by the grading induced by $\omega$, 
a direct adaptation of the proof of \cite[7.2.2]{Avramov6Lectures} allows us to conclude that
$$
\partial(X_1)\subseteq (\fn ,t)^2, \qquad
\partial(X_{n+1})\subseteq (\fn ,t)X_n+ \sum_{i=1}^{n-1} X_{i}\left(S[t][X]\right)_{n-i}
\quad\mbox{if }
n\gs 1
$$
hence the model $S[t][X]\otimes_{S[t]} S[t]/(t)$ is  minimal as a model of $S/\initial_{\omega} I$ over $S$.  
By Theorem \ref{Avramov}, we conclude that for every $i\gs 2$
$$\varepsilon_i(S/\initial_\omega I) = \Card(X_{i-1}) \gs \varepsilon_i(S/I).$$
\end{proof}

Serre showed that the following coefficientwise inequality of formal power series holds:
$$
P^R_k (z)  \preceq \frac{(1+z)^n}{1 - \sum_{i = 1}^n \beta_i^S(R)z^{i+1}}.
$$
The ring $R$ is called {\bf Golod} if the  equality is achieved;  
 we refer the reader to \cite[Chapter 5]{Avramov6Lectures} for a thorough treatment of Golod rings. 
If $I, J\subseteq (T_1,\ldots, T_n)^2$ are two homogeneous ideals such that $S/J$ is Golod and $\beta_i^S(S/I) \ls \beta_i^S(S/J)$ for all $i$, the inequality $P^{S/I}_k(z) \preceq P^{S/J}_k (z)$ holds (cf. \cite[proof of 1.3]{Peeva0Borel}). In the next proposition we extend this result to the sequence of deviations.  

\begin{prop}\label{cor Golod increase}
Let $S=k[T_1,\ldots,T_n]$ and $I, J\subseteq S$ be two homogeneous ideals.
If $S/I$ and $S/J$ are  Golod rings such that $\beta_i^S(S/I) \ls \beta_i^S(S/J)$ for all $i$, then
$\varepsilon_i(S/I) \ls \varepsilon_i(S/J)$
for every $i\gs 2$.
\end{prop}

\begin{proof}
Let $R$ be a Golod ring, $H^R$ the homology of the Koszul complex of $R$, and $\pi(R)$ the homotopy Lie algebra of $R$.
Then $\pi^{\gs 2}(R)$ is the free graded Lie algebra generated by the  vector space $\Hom_k(\Sigma H^R_{\gs 1},k)$,
and 
 $\varepsilon_i(R) = \dim_k \pi^i(R)$ for every $i$ (cf. Section \ref{Preliminaries}).
Since $\beta_i^S(R)=\dim_k H^R_i$ for every $i$,  the desired inequality follows.
\end{proof}

\begin{remark} If we remove the assumption that $S/I$ and $S/J$ are Golod rings then the result is not true.  Indeed, one can take $I = (T_1^2,T_1 T_2)$ and $J = (T_1^2,T_2^2)$.  While these ideals share the same Betti numbers over $S$, $\ee_3(S/I) = 1$ and $\ee_3(S/J) = 0$.
\end{remark}

Let $<$ be a term order on $S$.
It is well-known that there exists a non-empty Zariski open subset $\mathcal{U} \subseteq \Gl_n(k)$ such that the ideal $\initial_<(g\cdot I)$ is the same for every $g\in \mathcal{U}$, 
where $g\cdot I$ denotes the image of $I$ under the change of coordinates in $S$ defined by $g$ (cf. \cite[4.1.2]{HerzogHibi}). 
This  ideal is called the  {\bf generic initial ideal} of $I$ with respect to $<$ and is denoted by $\gin_< I$.
An important property of $\gin_<I$ is that it is fixed by the action of the Borel subgroup of $\Gl_n(k)$;
when $\ch \, k =0$ this means that $\gin_<I$ is strongly stable (cf. \cite[4.2.1, 4.2.6]{HerzogHibi}), 
thus $S/\gin_<I$ is a Golod ring by \cite[Theorem 4]{HerzogReinerWelker}.

Let $\HF_I$ denote the Hilbert function of  the ideal  $I$. 
The {\bf lex-segment ideal} of $I$ is 
 the vector space spanned by the lexicographically first $\HF_I(d)$ monomials of $S$ in each degree $d$.
 It was shown in \cite{Bigatti}, \cite{Hulett}, \cite{Pardue} that for any ideal $J$ with $\HF_J = \HF_L$  and every $i \gs 0$ we have  
 \begin{equation}\label{BHPInequality}
\beta_{i}(S/J) \ls \beta_{i}(S/L).
\end{equation}
Note that the ideal  $L$ is strongly stable, and therefore $S/L$ is a Golod ring.

We present now the second main result of this section. 
We provide two proofs;
although the second one only works in characteristic zero,  
we present both  in order to show an application of Proposition \ref{cor Golod increase}.

\begin{thm}\label{Main Theorem on Lex}
Let $S=k[T_1,\ldots,T_n]$ and $\fn=(T_1,\ldots,T_n)$.
Let   $I\subseteq \fn^2$ be a homogeneous ideal and $L$ be the lex-segment ideal of $I$. Then, for every $i\gs 1$,
$$\varepsilon_i (S/I) \ls \varepsilon_i (S/L).$$
\end{thm}

\begin{proof}[Proof 1.]

We follow the construction of Pardue in \cite{Pardue};
 our approach is  similar  to the one in the proof of the main result in \cite{PeevaConsec}. 
 Pardue's proof shows that any ideal in the Hilbert scheme is connected to the lex-segment ideal by a sequence of deformations and degenerations.  
 It suffices to show that at each step the deviations do not decrease. 
 The result is clear for $i=1$, so let $i\gs 2$.

\begin{enumerate}
\item {\it Generic changes of coordinates}. Since deviations depend only on the isomorphism class of $R$, generic changes of coordinates preserve deviations. 

\item {\it Passing to an initial ideal}. This follows by Theorem \ref{degenerations have bigger deviations}.

\item {\it Polarization and then factoring out generic hyperplane sections}.  
It is well known that an ideal $I$ and its polarization  are related via a regular sequence of linear forms.  Finally, the generic hyperplane sections that Pardue employs are always a regular sequence, 
so  by \cite[7.1.6]{Avramov6Lectures} the deviations do not decrease from the second one on.
\end{enumerate}
\end{proof}

\begin{proof}[Proof 2 ($\ch\, k = 0$).]
We may assume that $i \gs 1$.
Let $g$ be a generic change of coordinates. By Proposition \ref{degenerations have bigger deviations}, we have for all $i$
$$\varepsilon_i(S/I) = \varepsilon_i(S/(g\cdot I)) \ls \varepsilon_i(S/\initial_< (g \cdot I)) = \varepsilon_i(S/\gin_< I).$$
Since $S/(\gin_<I)$ and $S/L$ are Golod rings and $\HF_{\gin_<I}=\HF_I =\HF_L$, the conclusion follows from   \eqref{BHPInequality} and Proposition \ref{cor Golod increase}.
\end{proof}

\begin{remark}
We deduce  from  \eqref{PoincareDeviations} that $\beta_i^{S/I}(k)$ can be written as a function of the deviations $\varepsilon_i(S/I)$
using only sums, products, and binomial coefficients, see also \cite[7.1]{Avramov6Lectures}.
Thus, a pointwise inequality for the deviations implies a pointwise inequality for the Betti numbers of $k$.
In this way from Theorem  \ref{degenerations have bigger deviations} we recover  the well-known fact that the Poincar\'e series of the residue field can only grow larger when passing to an initial ideal:
$$P^{S/I}_k(z) \preceq P^{S/\initial_\omega I}_k (z)$$
\noindent
whereas from Theorem \ref{Main Theorem on Lex} it follows that considering the lex-segment ideal gives the largest Poincar\'e series among all ideals of $S$ with the same Hilbert series, 
a fact originally proved in \cite{Peeva0Borel}:
$$P^{S/I}_k(z) \preceq P^{S/L}_k (z).$$
\noindent
However, the next example shows that a pointwise inequality for the Betti numbers of $k$ does not necessarily imply a pointwise inequality for the deviations.
\end{remark}
\begin{example}
Let $S = k [ T_1, \ldots, T_5]$. 
Denote by $R_6$ and $R_8$ the quotients of $S$ by 6 and 8 generic quadrics, respectively; 
we determine their Poincar\'e series following \cite{Roos}.
The two rings\textbf{•} have the property $\mathcal{L}_3$ (cf. \cite[pp. 480, 482]{Roos}), and
thus their Poincar\'e series satisfy the equation
\begin{equation}\label{genKoszul}
\frac{1}{P^{R}_k(z) } = \frac{ 1 + \frac{1}{z}}{ \HS_A(z) } -  \frac{\HS_R(-z)}{z}
\end{equation}
where  $A$ is the subalgebra of the Yoneda algebra $\Ext_R(k,k)$ generated by $\Ext_R^1(k,k)$.
$A$ is
the enveloping algebra of a graded Lie algebra $\eta=\oplus_{i\gs1}\eta^i $ generated in degree 1. Therefore, its Hilbert series is given by
\begin{equation}\label{eqA1}
\HS_A (z) = \prod_{i=1}^\infty \frac{(1+z^{2i-1})^{e_{2i-1}}}{(1-z^{2i})^{e_{2i}}}
\end{equation}
where $e_i = \dim_\Bbbk \eta^i$ (notice that these are not the deviations of $R$).

The Hilbert series of the two rings are
$
\HS_{R_6}(z) = 1 +5z +9 z^2 +5 z^3
$
and
$
\HS_{R_8}(z) = 1+ 5z + 7z^2.
$
In order to determine the $e_i$, we proceed as in \cite[p. 487]{Roos}. 
We only show the computations for the case of 6 quadrics, the other case being analogous. 
The  series $\HS_A(z)$ is equal to the diagonal power series 
\begin{equation}\label{eqA2}
\sum_{i\gs 0} \beta_{ii}^{R_6}(\Bbbk)z^i = 1 + 5z + 16 z^2 + 40 z^3 + 86 z^4 + 166 z^5 + 296 z^6 + \cdots.
\end{equation}
Comparing \eqref{eqA1} and \eqref{eqA2} we find that $e_1 = 5$, so we divide  \eqref{eqA2} by $(1+z)^5$ 
\begin{equation}\label{eqA3}
\frac{\HS_A(z)}{(1 + z)^5}=1 + 6z^2 + 21 z^4 + 56 z^6 + \cdots
\end{equation}
from which we find that $e_2 = 6$. 
Now we multiply \eqref{eqA3} by $(1-z^2)^6$ and get
\begin{equation}
\frac{(1-z^2)^6\HS_A(z)}{(1 + z)^5} = 1 + 0z + 0z^2 + 0z^3 + \cdots.
\end{equation}
Hence $\eta^3 = 0$ and, since $\eta$ is generated in degree 1, we conclude that $\eta^i =0 $ for all $i \gs 3$.

From  \eqref{genKoszul} we obtain 
$$
P^{R_6}_k(z)= \frac{ 1 }{(z-1)(z^6-3 z^5+z^4+5 z^3-5 z^2+4 z-1)},
$$
and in the same way we obtain 
$$
P^{R_8}_k(z) = \frac{-(z^2-z+1)^5 }{d(z)},
$$
\begin{eqnarray*}
\textrm{where } d(z)&=& z^{18}-9 z^{17}+37 z^{16}-93 z^{15}+160 z^{14}-192 z^{13}+136 z^{12}+31 z^{11}\\
&-&270 z^{10}+ 505 z^9-664 z^8+710 z^7-636 z^6+479 z^5-294 z^4+140 z^3\\
&-&47 z^2+10 z-1.
\end{eqnarray*}
Using these expressions we compute the first few deviations and find 
$$
\varepsilon_4(R_6) = 16 > 9 = \varepsilon_4(R_8)
$$
whereas we verify that 
$
\beta^{R_6}_i(k) \ls \beta^{R_8}_i(k)$
for $i = 0, \ldots, 9$.
Finally, 
using partial fraction decompositions for $P^{R_6}_k(z)$ and $P^{R_8}_k(z)$
we obtain estimates for the Betti numbers, showing that $\beta^{R_6}_i(k) \ls \beta^{R_8}_i(k)$
for all $i \gs 10$, and thus
$$
\beta^{R_6}_i(k) \ls \beta^{R_8}_i(k)
\quad 
\mbox{for all }
i \gs 0.
$$ 
\end{example}

\section{Exponential Growth of Deviations}

The goal of this section is to prove that, when $R$ is not a complete intersection, deviations grow exponentially if $R$ is Golod (cf. Theorem \ref{goloddev})
or if $R$ is a Koszul algebra with the same Hilbert series as a ring presented by the edge ideal of a claw-free graph (cf. Theorem \ref{clawfreedev}).

We now briefly introduce Koszul algebras, referring to \cite{Conca} for a survey on the topic. 
The ring $R$ is a {\bf Koszul algebra} if the minimal free $R$-resolution of $k$ is linear, i.e., $\be{i}{j}^R(k)=0$ whenever $i\neq j$.
If $R$ is Koszul then $I$ is generated by quadrics;
however, these two conditions  are not equivalent. 
For instance, if  $I=(T_1^2,T_2^2,T_3^2, T_4^2, T_1T_2+T_3T_4)$, then $\be34^R(k)\neq 0$. 
By a classical theorem of Fr\"oberg, $R$ is Koszul if $I$ is a quadratic monomial ideal (cf. \cite{Froberg}). Denote the Hilbert series of $R$ by $\HS_R(\zi)=\sum_{i}\dim_k(R_{i})\zi^{i}$; then $R$ is Koszul if and only if the following identity occurs:
 
\begin{equation}\label{PoncareHilbert}
P^R_k(z)\HS_R(-z)=1.
\end{equation}

The following proposition provides a compact formula for deviations  when the Poincar\'e series has a certain form.

\begin{prop}\label{mobius}
Let $S=k[T_1,\ldots, T_n]$ and $R=S/I$ with $I\subseteq \fn^2$ a homogeneous ideal.
Assume there exist $c\in \NN$ and complex numbers $\{\alpha_j\}_{1\ls j\ls m}$ so that 
$P^R_k(z)=\frac{(1+z)^c}{\prod_{j=1}^{m}(1+\alpha_jz)}$.  
Then for every $i\gs 2$ \[\ee_i(R) = \frac{(-1)^{i}}{i}\sum_{d \mid i}\mu\Big(\frac{i}{d}\Big)\sum_{j=1}^{m}\alpha_{j}^d\] where $\mu$ is the M\"obius function. 
\end{prop}

\begin{proof}
Set $\ee_i:=\ee_i(R)$. 
From   \eqref{PoincareDeviations} and the assumption we obtain 
\begin{equation}\label{tobln}
\prod_{i=1}^{\infty}(1-z^i)^{(-1)^{i+1}\ee_i}=\frac{(1-z)^c}{\prod_{j=1}^{m}(1-\alpha_jz)}.
\end{equation}
Proceeding as in \cite[p. 23]{QuadAlg}, we apply natural logarithm at both sides of  \eqref{tobln}, and compare the coefficient of $z^i$ of the corresponding Maclaurin series to obtain $\sum_{d | i}{(-1)^{d}}d\ee_d=-c+\sum_{j=1}^m\alpha_j^i.$
The result now follows by applying the M\"obius inversion formula to the arithmetic function $f(i)=(-1)^{i}i\ee_{i}$ and using the fact that $\sum_{d \mid i} \mu(d) = 0$ for every $i\gs 2$.
\end{proof}

Note that the assumption of Proposition \ref{mobius} is satisfied by several classes of rings, including complete intersections (cf. \cite{Tate}), monomial rings (cf. \cite{Backelin}),  Golod rings (cf. Section \ref{Extremal}), and Koszul algebras,
as it follows from  \eqref{PoncareHilbert}.
Given two numerical sequences $\{a_i\}_{i\gs1}, \{b_i\}_{i\gs1}$, 
the expression $a_i \sim b_i$ stands for asymptotic equality, i.e.  $\lim_{i\to \infty} \frac{a_i}{b_i}= 1$.

\begin{example}\label{exampMobius}
Consider a Koszul algebra $R$ with $h$-polynomial $h(z) = 1 + mz$ for some $m\in \NN$. 
This holds for instance for $I = (T_1, \ldots, T_m)^2\subseteq S= k[T_1, \ldots, T_m]$ or 
$I=(T_iT_j \, | \, 1 \ls i < j \ls m+1)\subseteq S=k[T_1, \ldots, T_{m+1}] $ (the edge ideal of the complete graph $\mathcal{K}_{m+1}$). 
By Proposition \ref{mobius} for each $i \gs 2$ we have 
\[\ee_i(R) = \frac{(-1)^{i}}{i}\sum_{d \mid i}\mu\Big(\frac{i}{d}\Big)(-m)^d \sim \frac{m^i}{i}.\]
\end{example}

We present next the first main result of this section. We show that deviations grow exponentially for Golod rings that are not complete intersections. For similar techniques applied to Betti numbers, see \cite{Sun}.

\begin{thm} \label{goloddev}
Let $S=k[T_1,\ldots,T_n]$ and $\fn=(T_1,\ldots,T_n)$. 
Let $I$ be a non-principal homogeneous ideal such that $I\subseteq \fn^2$ and $R=S/I$. 
If $R$ is a Golod ring, then there exists a real number $\rho>1$ so that $$\ee_i(R)\sim \frac{\rho^i}{i}.$$
\end{thm}
\begin{proof}
Let $r$ be the radius of convergence of $P_k^R(z)$. The ring $R$ is not regular, hence $r\ls 1$ as the power series $P_k^R(z)$ is infinite. 
Since $R$ is Golod and not a hypersurface, it is not a complete intersection (see \cite[p. 47, Remark]{Avramov6Lectures}), therefore $0<r<1$ by \cite[4.2.3, 8.2.2]{Avramov6Lectures}.  Moreover, since $R$ is Golod, we have $P_k^R(z) = \frac{(1+z)^n}{v(z)}$ inside $B_{\CC}(0, r)$, where $v(z)=1 - \sum_{i = 1}^n \beta_i^S(R)z^{i+1}$. 

Now, since  $P_k^R(z)$ is a rational function, it has singularities with absolute value $r$, and since it has nonnegative real coefficients, we conclude  that $r$ is a singularity of $P_k^R(z)$. Hence, $r$ is a root of $v(z)$. 
We will show that $r$ is a simple root of $v(z)$ and  that it is the only root on the boundary of $B_{\CC}(0, r)$. 
The result will follow by Proposition \ref{mobius} as $\alpha_1=-\frac{1}{r}$ and $|\alpha_j|<|\alpha_1|$ otherwise. 

Set $\beta_i := \beta_i^S(R)$. Since $I$ is not principal, we have $\beta_1,\, \beta_2 > 0$. Therefore $r$ is a simple root of $v(z)$ as  $v'(r) = -\sum_{i=1}^n (i+1)\beta_i r^i < 0$. If $y\neq r$ and $|y|=r$, then $y\not\in\RR_{\gs 0}$ and hence 
$$|\beta_1y^2 + \beta_2y^3|=r^2|\beta_1+\beta_2y| < r^2(\beta_1 + \beta_2 r)=\beta_1r^2+\beta_2r^3.$$ 
Thus \[\Big| \sum_{i=i}^n \beta_iy^{i+1}\Big| \ls |\beta_1y^2 + \beta_2y^3| + \Big|\sum_{i=3}^n \beta_iy^{i+1}\Big| < \sum_{i=1}^n \beta_i r^{i+1} = 1,\] i.e. $y$ is not a root of $v(z)$. 
The proof is completed.
\end{proof}

Let $\G$ be a graph with vertices $\{v_1,v_2,\ldots, v_n\}$. Denote by  
$$
I(\G)=(T_iT_j\,|\, v_iv_j \text{ is an edge of }\G)\subseteq S
$$ 
its  edge ideal.
By Fr\"oberg's theorem, an algebra $R$ presented by an edge ideal is Koszul. 

A graph with four vertices and edges $\big\{\{v_1,v_2\},\,\{v_1,v_3\},\,\{v_1,v_4\}\big\}$ is called a {\bf claw} (see Figure \ref{figClaw}). 
A simple graph $\G$ is said to be {\bf claw-free} if no claw appears as an induced subgraph of $\G$. 
Note that complete graphs are claw-free
(cf.  \eqref{exampMobius}).

\begin{figure}[h]
\centering
$$\xymatrix@R=2mm{& \bullet  \ar@{-}[rdd] \ar@{-}[ldd] \ar@{-}[dd] \\ & &   \\ \bullet & \bullet & \bullet }$$
\caption{A claw.}
\label{figClaw}
\end{figure}

The following theorem shows that the asymptotic behavior of deviations observed in Example \ref{exampMobius} holds more generally.

\begin{thm} \label{clawfreedev}
Let $S=k[T_1,\ldots,T_n]$ and  $R=S/I(\mathcal{G})$ where $\mathcal{G}$ is a graph. 
If $\G$ is claw-free 
then there exist $b\in \NN$ and $\rho \in \mathbb{R}$ with $\rho>1$ so that $$\ee_i(R)\sim \frac{b\rho^i}{i}.$$ 
\end{thm}
\begin{proof}
Let $r$ be the radius of convergence of $P_k^R(z)$. 
By  \eqref{PoncareHilbert} we have that $P_k^R(z)=\frac{(1+z)^c}{h(-z)}$ inside $B_{\CC}(0, r)$, where $h(z)$ is the $h$-polynomial of $R$ and $c=\dim(R)$. Since $R$ is not a complete intersection, we conclude as in the proof of Theorem \ref{goloddev} that $0<r<1$ and $-r$ is a root of $h(z)$; let $b$ be the multiplicity of this root. The result will follow as in the proof of Theorem \ref{goloddev} via Proposition \ref{mobius} once we show that $-r$ is the only root of $h(z)$ on the boundary of $B_{\CC}(0, r)$. 

Let $\Delta$ be the independence complex of $\G$, i.e., the simplicial complex having $I(\G)$ as its Stanley-Reisner ideal. By \cite[1.1]{ChuSey}, $f(z)$, the $f$-polynomial of $\Delta$, has only real roots because $\G$ is a claw-free graph; hence so does $h(z)$ as these two polynomials satisfy the relation $h(z)=(1-z)^cf(\frac{z}{1-z})$ for every $z\neq 1$. Since $r$ is not a root of $h(z)$ (see \cite[4.8]{Uliczka}), the result follows.
\end{proof}

\begin{remark}
Theorem \ref{goloddev} provides another class of graphs $\mathcal{G}$ whose deviations grow exponentially, namely those whose complementary graph is chordal. 
By \cite{FrobergGraphs}, these graphs $\mathcal{G}$ are precisely the ones such that $I(\mathcal{G})$ has a linear resolution (note that in this case the characteristic of the field $k$ is not relevant), and thus $S/I(\mathcal{G})$ is Golod by \cite[Theorem 7]{BackelinFroberg}.
\end{remark}

\begin{example}
We  now examine in detail the deviations for  paths $\mathcal{P}_n$ and cycles $\mathcal{C}_n$ on $n$ vertices, with $n\gs 3$ (see also \cite[2.6]{BDGMS}). 
Consider the $f$-polynomial and the $h$-polynomial of the  independence complex of these graphs. 
The roots of $f$ for  $\mathcal{C}_n$ and $\mathcal{P}_n$ are determined explicitly in \cite{AlikhaniPeng} and are respectively
\begin{align*}
c_s^{(n)} &= -\frac{1}{2(1+\cos(\frac{2s-1}{n}\pi))} & s &= 1, 2, \ldots, \Big\lfloor\frac{n}{2}\Big\rfloor \\
p_s^{(n)} &= -\frac{1}{2(1+\cos(\frac{2s\pi}{n+2}))} & s &= 1, 2, \ldots, \Big\lfloor\frac{n+1}{2}\Big\rfloor.
\end{align*}
From  the relation between the $f$-polynomial and the $h$-polynomial
we have
\[
\tau\ne -1 \text{ is a root of } f \Leftrightarrow \frac{\tau}{1+\tau}\text{ is a root of } h.
\]
Note that $h(1) \neq 0$.
For each  $s, n$ we have $c_s^{(n)}, p_s^{(n)}<0$  
and in both cases $s=1$ gives the root  $\tau$ of $f$ with minimum modulus.
Note  that $\tau \in [-\frac{1}{2}, 0)$. 
Let $\sigma \not\in \{\tau, -1\}$ be another root of $f$.
If $\sigma < -1$ then $\lvert\frac{\sigma}{1+\sigma}\rvert > 1 \gs \lvert\frac{\tau}{1+\tau}\rvert$, 
while if $-1 < \sigma < \tau$ then $\lvert\frac{\sigma}{1+\sigma}\rvert > \lvert\frac{\tau}{1+\tau}\rvert$
because the map $x \mapsto \lvert\frac{x}{1+x}\rvert$ is strictly decreasing on $(-1, 0]$. 
It follows that the root of minimum modulus of $h$ is  the image of the root of minimum modulus of $f$ via the map $x \mapsto \frac{x}{1+x}$.
For both 
$\mathcal{C}_n$ and $\mathcal{P}_n$,
this root  converges to $-\frac{1}{3}$  as $n$ goes to infinity.
Applying Theorem \ref{clawfreedev} we get that 
\[
\begin{split}
\ee_i(S/I(\mathcal{C}_n)) &\sim \frac{(2(1+\cos(\frac{\pi}{n})) - 1)^i}{i},
\\ 
\ee_i(S/I(\mathcal{P}_n)) &\sim \frac{(2(1+\cos(\frac{2\pi}{n+2})) - 1)^i}{i}.
\end{split}
\] 
As $n$ goes to infinity, both expressions on the RHS approach $\frac{3^i}{i}$.
\end{example}

We observe that the integer $b$ in Theorem \ref{clawfreedev} can be arbitrary.

\begin{example}\label{anyb}
If $\G$ consists of $m$ connected components, each of them being a copy of $\mathcal{P}_4$, then $h(z) =(1+2z)^m$  and hence  $b=m$ in Theorem \ref{clawfreedev}. 
\end{example}

\begin{remark}Theorem \ref{clawfreedev} can also be used to produce examples of non-monomial Koszul algebras with exponential growth of deviations:
by  \eqref{PoncareHilbert}, two Koszul algebras with the same $h$-polynomial must have the same deviations for $i\gs 2$.
\end{remark}

We conclude by asking a combinatorial question. 

\begin{question}\label{singleroot}
Let $\G$ be a graph and let $h$ be the $h$-polynomial of $S/I(\G)$. Does $h$ have a unique (but not necessarily simple) root of minimum modulus?
\end{question}

Note that, since we already know that there exists a real negative root of $h$ of minimum modulus, an affirmative answer to Question \ref{singleroot} would imply the exponential growth of deviations of $S/I(\G)$ for any graph $\G$ such that $S/I(\G)$ is not a complete intersection.
Moreover, it is already known that the $f$-polynomial of the independence complex of $\G$ admits a unique (but not necessarily simple) root of minimum modulus (cf. \cite{GoldwurmSantini}, \cite{Csikvari}).

\section*{Acknowledgements}

We would like to thank the organizers, lecturers, and participants of Pragmatic 2014  where this project started. 
We thank especially Srikanth Iyengar and Aldo Conca for introducing us to the topic and for the enlightening discussions. 
The last four authors also thank their Ph.D. advisors, Aldo Conca, Craig Huneke, Bernd Ulrich, and Giulio Caviglia, respectively.
We acknowledge the use of the software \texttt{Macaulay2} \cite{Macaulay2}. The authors are also grateful to the referee for her or his helpful suggestions.

\end{document}